\documentclass[11pt, reqno]{amsart}
\usepackage{fullpage}
\usepackage{comment}
\usepackage{enumitem, amssymb,amsmath,amsthm,amscd,mathrsfs,graphicx, color}
\usepackage{xr-hyper}
\usepackage{hyperref}

\usepackage[cmtip,all,matrix,arrow,tips,curve]{xy}
\usepackage[usenames,dvipsnames]{xcolor}
\usepackage{xypic}
\hypersetup{colorlinks=true,citecolor=ForestGreen,linkcolor=Maroon,urlcolor=NavyBlue}
\usepackage{tikz-cd}
\usetikzlibrary{babel}
 \usepackage{mathtools}
\usepackage{mathpazo}

 \newif\ifHideFoot
\HideFoottrue   \HideFootfalse

\setlist[enumerate]{leftmargin=8.5mm}

\numberwithin{equation}{section}

 \usepackage{zref-clever}
\zcsetup{cap,abbrev}

\newcommand{\cref}[1]{\zcref{#1}}
\newcommand{\Cref}[1]{\zcref[S]{#1}}

\zcRefTypeSetup{convention}{Name-sg = Convention}
\zcRefTypeSetup{teo}{Name-sg = Theorem}
\zcRefTypeSetup{pro}{Name-sg = Proposition}
\zcRefTypeSetup{lem}{Name-sg = Lemma}
\zcRefTypeSetup{cor}{Name-sg = Corollary}
\zcRefTypeSetup{wrn}{Name-sg = Warning}
\zcRefTypeSetup{condition}{Name-sg = Condition}
\zcRefTypeSetup{rem}{Name-sg = Remark}

  \NewDocumentCommand{\newzctheorem}{mO{#1}m}{
  \newtheorem{#1}[sharedtheoremcounter]{#3}
    \AddToHook{env/#1/begin}{       \zcsetup{countertype={sharedtheoremcounter=#2}}}
}

\newzctheorem{teo}{Theorem}
\newzctheorem{pro}{Proposition}
\newzctheorem{lem}{Lemma}
\newzctheorem{cor}{Corollary}
\newzctheorem{wrn}{Warning}
\newzctheorem{condition}{Condition}

\newtheorem{teoalpha}{Theorem}

\zcsetup{countertype={teoalpha=theorem}}

\theoremstyle{remark}
\newzctheorem{rem}{Remark}

\newenvironment{alphabetize}{\begin{enumerate}[label=(\alph*)]
 }{\end{enumerate}}

\renewcommand{\MR}[1]{}

\ifHideFoot

\newcommand{\Yano}[1]{}
\newcommand{\Jeff}[1]{}
\newcommand{\Charles}[1]{}

\else

\newcommand{\marg}[1]{\normalsize{{
   \color{red}\footnote{{\color{blue}#1}}}{\marginpar[\vskip
   -.25cm{\color{Maroon}\hfill\thefootnote$\implies$}]{\vskip
    -.2cm{\color{Maroon}$\impliedby$\tiny\thefootnote}}}}}
\newcommand{\Yano}[1]{\marg{(Yano) #1}}
\newcommand{\Jeff}[1]{\marg{(Jeff) #1}}
\newcommand{\Charles}[1]{\marg{(Charles) #1}}

\fi

\newcommand{\rest}[1]{|_{#1}}
\newcommand{\til}[1]{{\widetilde{#1}}}

\def\inv{^{-1}}

\def\proj{{\mathbb P}}
\def\cx{{\mathbb C}}

\def\rat{{\mathbb Q}}
\def\integ{{\mathbb Z}}

\def\iso{\cong}

\renewcommand{\bar}[1]{{\overline{#1}}}

\DeclareMathOperator{\aut}{Aut}
\DeclareMathOperator{\alb}{Alb}

\DeclareMathOperator{\gal}{Gal}

\DeclareMathOperator{\spec}{Spec}
\DeclareMathOperator{\pic}{Pic}

\DeclareMathOperator{\Ab}{Ab}
\DeclareMathOperator{\A}{A}

\DeclareMathOperator{\chow}{CH}

\DeclareMathOperator{\sym}{Sym}
\DeclareMathOperator{\bwchow}{{\bf CH}}
\DeclareMathOperator{\wc}{WC}

\DeclareMathOperator{\albmorphism}{alb}
\newcommand{\st}[1]{\left\{#1\right\}}

\def\tplus{\boxplus}

 \title{Regular homomorphisms, with a twist}

\author{Jeffrey D. Achter}
\address{Colorado State University, Department of Mathematics,
 Fort Collins, CO 80523,
 USA}
\email{j.achter@colostate.edu}

\author{Sebastian Casalaina-Martin }
\address{University of Colorado, Department of Mathematics,
 Boulder, CO 80309, USA }
\email{casa@math.colorado.edu}

\author{Charles Vial}
\address{Universit\"at Bielefeld, Fakult\"at f\"ur Mathematik, Germany}
\email{vial@math.uni-bielefeld.de}

\thanks{Research of the first and second authors is supported in part by grants from the Simons Foundation (637075 and 581058, respectively). Research of the third author is supported by the Deutsche Forschungsgemeinschaft (DFG, German Research Foundation) – Project-ID 491392403 – TRR 358.
}

\begin{document}

  \begin{abstract}
 Let $X/K$ be a variety over a field, and $A/K$ an abelian variety.  A regular homomorphism to $A$ (in codimension $i$)  induces, for every smooth geometrically connected pointed $K$-scheme $(T,t_0)$ and every cycle class $Z \in \chow^i(T\times X)$, a morphism $T \to A$ of varieties over $K$.
 In this note we show that,  even if $T$ admits no $K$-point, the data $(T,Z)$ determines a torsor $A^{(T,Z)}$ over $K$ under $A$ and a $K$-morphism $T \to A^{(T,Z)}$.  
 This can be used to provide an obstruction to the existence of algebraic cycles defined over $K$.  
We then connect this obstruction to some recent results of Hassett--Tschinkel and Benoist--Wittenberg on rationality of threefolds.   
 \end{abstract}

  \maketitle
    \setcounter{tocdepth}{1}

\section{Introduction}
 
Let $K$ be a field with separable closure $\bar K$. We will use the convention that a \emph{variety} over $K$ is a geometrically reduced separated scheme of finite type
over $K$. 
For  a smooth proper
 variety $X$ over~$K$, abelian varieties can be used to understand $\A^i(X_{\bar K})$, the group of rational equivalence classes of  algebraically trivial cycles on $X_{ \bar K}$ of codimension $i$,  via the notion of a regular homomorphism. 

Recall that if  $A/K$ is an abelian variety, then a regular homomorphism over $K$, or Galois-equivariant regular homomorphism, is a group homomorphism $\phi\colon \A^i(X_{\bar K}) \to A(\bar K)$ with the following property:  For any smooth connected $K$-pointed variety
  $(T,t_0)$ over $K$, and any cycle class $Z \in \chow^i(T\times X)$, the map of sets
\[
\xymatrix{
T(\bar K) \ar[r] & \A^i(X_{\bar K}) \ar[r]^\phi & A(\bar K)\\
t \ar@{|->}[r] & [Z_t]-[Z_{t_0}] \ar@{|->}[r] & \phi([Z_t]-[Z_{t_0}])
}
\]
is induced by a morphism $\psi_Z = \psi_{(T,t_0,Z)}: T \to A$ of schemes over $K$. 
 (This concrete formulation of the functorial definition in \cite{ACMVfunctor} is both closer to the classical definition and better-suited to our purposes here.) 
An algebraic representative for $X$ (in codimension $i$) is an abelian variety $\Ab^i_{X/K}$, and a regular homomorphism $\A^i(X_{\bar K}) \to \Ab^i_{X/K}(\bar K)$, which is initial for all such maps.

At this point, the reader should keep in mind two slightly different kinds of examples of regular homomorphisms.
On one hand, $\Ab^1_{X/K}$ and $\Ab^{\dim X}_{X/K}$ always exist;
these are, respectively, $(\pic^0_{X/K})_{\operatorname{red}}$ and $\alb_{X/K}$.  Moreover,
 $\Ab^2_{X/K}$ exists, as well; this is due to Murre \cite{murre83} if $K$ is algebraically closed and to \cite[Thm.~6.1]{ACMVfunctor} in general.
On the other hand, suppose $K$ is a subfield of $\cx$.  By \cite[Thm.~A]{ACMVdmij} (see also \cite[Thm.~9.1]{ACMVfunctor}), 
the algebraic intermediate Jacobian $J_a^{2i-1}(X_\cx)$, which is by definition the image of the Abel--Jacobi map restricted to algebraically trivial cycle classes, admits a distinguished model $J^{2i-1}_{a,X/K}$ over $K$, and the Griffiths Abel--Jacobi map descends to a regular homomorphism $\A^i(X_{\bar K}) \to J^{2i-1}_{a,X/K}(\bar K)$.

Recently, \emph{torsors} under such abelian varieties have been used to detect  irrationality of geometrically rational varieties.  
For example, suppose $X$ is a smooth projective variety over  a subfield $K$ of $\cx$, and let $J = J^{2i-1}_{a,X/K}$.  Given a geometrically irreducible component $T$ of the Chow scheme of codimension-$i$ cycles on $X$, Hassett and Tschinkel construct \cite[Thm.~4.5]{hassetttschinkel21} a torsor $J^T$ under $J$; and they further show that this construction is compatible with addition in Chow, i.e., that in the Weil--Ch\^atelet group of $J$ one has $[J^{T_1}] + [J^{T_2}] = [J^{T_1\times T_2}]$. 
If now $X$ is a smooth projective geometrically rational threefold over an arbitrary field $K$,  Benoist and Wittenberg construct a codimension-2 Chow scheme $\bwchow^2_{X/K}$ that represents a functor which is a certain subquotient of K-theory.  
Its connected component of the identity $(\bwchow^{2}_{X/K})^\circ$ is an abelian variety -- isomorphic to $\Ab^2_{X/K}$, if $K$ is perfect -- and so its other geometrically irreducible components are torsors under that abelian variety.  
Benoist and Wittenberg \cite{benoistwittenberg23} and Hassett and Tschinkel \cite{hassetttschinkel21} use these torsors to construct an obstruction to the rationality of the smooth complete intersection of two quadrics in $\proj^5$.
Subsequently, Frei \emph{et al.} studied extensions of, and limitations to, this so-called intermediate Jacobian torsor obstruction, especially for certain conic bundles over $\proj^2$ \cite[Thms.~1.4 and~1.5]{freietal24}.

With this backdrop, we can finally explain the goal of the present
note.  It turns out that the various torsors constructed in
\cite{hassetttschinkel21} and \cite{benoistwittenberg23} require
neither (complex) intermediate Jacobians nor a functorial Chow scheme.
Instead, they arise from an arbitrary regular homomorphism. 

\begin{teoalpha}\label{T:main}
Let $X/K$ be a smooth proper variety, let $A/K$ be an abelian variety, and let $\phi:\A^i(X_{\bar K}) \to A(\bar K)$ be a regular homomorphism over $K$.
\begin{alphabetize}
\item Let $T/K$ be a smooth geometrically connected scheme, and let $Z \in \chow^i(T\times X)$.
Then there exists a torsor $A^{(T,Z)}$ under $A$ such that any choice of $\bar K$-point $t_0 \in T(\bar K)$ induces a ($K$-rational) morphism $T \to A^{(T,Z)}$ which, after base change to $\bar K$, agrees with $\psi_{(T_{\bar K}, t_0, Z_{\bar K})}$.

\item \label{T:mainb} Let $T_1$ and $T_2$ be smooth geometrically connected schemes over $K$, and let $Z_j \in \chow^i(T_j\times X)$.  Then there is an equality of torsor classes
\[
[A^{(T_1,Z_1)}] + [A^{(T_2,Z_2)}] = [A^{(T_1\times T_2,Z_1\tplus Z_2)}],
\]
where addition takes place in the Weil--Ch\^atelet group $\wc(A/K) \iso H^1(K,A)$, and $Z_1\tplus Z_2 = p_{13}^*Z_1 + p_{23}^*Z_2 \in \chow^i(T_1\times T_2 \times X)$.
\end{alphabetize}
\end{teoalpha}

There is a variant of this which works when the parametrizing scheme $T$ is not geometrically connected; we work out the details in \Cref{L:geomdisconnected}.

In work in progress, the authors use the ideas of \S\ref{S:torsors}
to construct a \emph{big algebraic representative} for $X$.  This is a group scheme
whose connected component of the identity is the (usual) algebraic
representative, and which supports an action by $\aut(X)$.
In~\S\ref{S:rationality}, we explain how these torsors  give a conceptual framework for understanding (failures of) rationality and existence of algebraic cycles for smooth projective varieties over a field, and we close by revisiting the rationality of smooth complete intersections of quadrics in $\proj^5$.

\subsection*{Acknowledgments}

The first-named author thanks the American Institute of Mathematics
for the opportunity to participate in the March 2024 workshop
\emph{Degree $d$ points on algebraic surfaces}, and Nathan Chen for
conversations which sparked this note. We all thank the referees for
their helpful suggestions.

\section{Torsors under algebraic representatives}
\label{S:torsors}

\subsection{Torsors}
\label{SS:torsors}

  Let $X/K$ be a smooth proper variety; we investigate codimension $i$ cycles on $X_{\bar K}$.  Let $A/K$ be an abelian variety, and as in the introduction let $\phi: \A^i(X_{\bar K}) \to A(\bar K)$ be a regular homomorphism over $K$. 
                             
  If $T_1$ and $T_2$ are both smooth varieties over $K$, and if $Z_j \in \chow^i(T_j \times X)$ for $j = 1,2$, we define
  \[
    Z_1 \tplus Z_2 = p_{13}^* Z_1 + p_{23}^* Z_2 \in \chow^i(T_1\times T_2 \times X).
  \]
  If $Z \in \chow^i(T\times X)$ then, following \cite[\S 7.1]{ACMVdmij}, we set
  \begin{equation}
    \label{E:deftilZ}
    \til Z =  Z \tplus (-Z) \in \chow^i(T\times T \times X).
  \end{equation}
  Every fiber of $\til Z$ is algebraically trivial.  Indeed, for $\bar K$-points $t$ and $t'$ of $T$ we have $\til Z_{t,t'} = Z_t - Z_{t'} \in \A^i(X_{\bar K})$.

  Whereas cycle classes parametrized by $K$-\emph{pointed} schemes give rise to morphisms to $A$, in general, we obtain a map to a certain torsor under $A$, as follows.
  
  Before proceeding, recall that if $B/K$ is an abelian variety, then $\wc(B/K)$, the Weil--Ch\^atelet group of $B$ over $K$, is an abelian group whose elements are isomorphism classes of torsors under $B$ over $K$.  (If the base field $K$ is clear from context, we will sometimes just write $\wc(B)$.) It is isomorphic to the Galois cohomology group $H^1(K, B) \coloneqq H^1(\gal(K),B(\bar K))$.  If $\beta: \gal(K) \to B(\bar K)$ represents a cohomology class, we may describe the corresponding torsor $B[\beta]$ by adopting the viewpoint that $B[\beta]$ ``is'' $B$ equipped with a twisted action of $\gal(K)$ on $B(\bar K)$, as in   \cite[\S
    5.2]{serreGC}.     For $b \in B(\bar K)$ and
  $\sigma\in \gal(K)$, denote the image of $b$ under $\sigma$ by
  $b^\sigma$; and let $b^{\til \sigma}$ denote its image under the
  action of $\gal(K)$ twisted by $\beta$.  Then we have
  \begin{equation}
  \label{E:twistedgalois}
  b^{\til\sigma} = \beta_\sigma +_{B(\bar K)} b^\sigma.
  \end{equation}

  \begin{proof}[Proof of \Cref{T:main}]
    We start by verifying that the function
      \[
  \xymatrix{
    \gal(K) \ar[r]^{\alpha = \alpha(T,t_0,Z)} & A(\bar K) \\
    \sigma \ar@{|->}[r] & \alpha_\sigma \coloneqq \phi(Z_{t_0^\sigma} -
    Z_{t_0})
  }
\]
is a one-cocycle, and thus its class in $H^1(K,A)$  determines a torsor $A^{(T,t_0,Z)}$ over $K$ under $A$.  We compute that
    \begin{align*}
      \alpha_\sigma &= \phi(\til Z_{t_0^\sigma,t_0})\\
      (\alpha_\sigma)^\tau &= \phi(\til Z_{t_0^\sigma,t_0})^\tau \\
      \intertext{and, by Galois-equivariance of $\phi$, we have}
    (\alpha_\sigma)^\tau      &= \phi((\til Z_{t_0^\sigma,t_0}      )^\tau) \\
      &= \phi(\til Z_{t_0^{\sigma\tau}, t_0^\tau})
      \intertext{because $Z$ is defined over $K$. We verify the cocycle condition by computing}
          (\alpha_\sigma)^\tau +\alpha_\tau &= \phi(\til
      Z_{t_0^{\sigma\tau},t_0^\tau} + \til Z_{t_0^\tau,t_0}) \\
     & = \phi(\til Z_{t_0^{\sigma\tau},t_0}) = \alpha_{\sigma\tau}.
    \end{align*}
    We now show that $\psi \coloneqq \psi_{(T_{\bar K}, t_0, Z_{\bar K})}:T_{\bar K} \to A_{\bar K}$ descends to a $K$-rational morphism $T \to A^{(T,t_0,Z)}$.  Recall that $A^{(T,t_0,Z)}$ is $A$ equipped with a twisted action of $\gal(K)$ \eqref{E:twistedgalois}, and
                          consider the morphism
  \[
  \xymatrix{
    T_{\bar K} \ar[rr]^{\psi} & & A_{\bar K}
  }
  \]
  of varieties over $\bar K$.  Identifying $(A_{\bar K},0_{A_{\bar K}})$ with $(A^{(T,t_0,Z)}_{\bar K},\psi(t_0))$, and
  in particular $A(\bar K)$ with $A^{(T,t_0,Z)}(\bar K)$, we have a morphism
  \[
  \xymatrix{
    T_{\bar K} \ar[r]^{\psi\quad } & A^{(T,t_0,Z)}_{\bar K} }
\]
which on  $\bar K$-points is given by 
  \[
\xymatrix{
     t\ar@{|->}[r] & \phi(Z_t-Z_{t_0}).
  }
  \]
To show that  $\psi$ descends to a morphism $T \to A^{(T,t_0,Z)}$
  over $K$, it suffices to show that $\psi$ is $\gal(K)$-equivariant on $\bar K$-points. This follows from the calculation that, for $\sigma \in \gal(K)$ and $t\in T(\bar K)$, we have 
  \begin{align*}
    \psi(t)^{\til \sigma} &= \alpha_\sigma + \psi(t)^\sigma \\
    &= \phi(Z_{t_0^\sigma} - Z_{t_0}) + \phi(Z_t - Z_{t_0})^\sigma \\
    &= \phi(Z_{t_0^\sigma} - Z_{t_0}) + \phi(Z_{t^\sigma} -
    Z_{t_0^\sigma}) \\
    &= \phi(Z_{t^\sigma} - Z_{t_0}) \\
    &= \psi(t^\sigma).
  \end{align*}
  To show the independence of the choice of geometric base point, let $t_1 \in T(\bar K)$ be any other point, with
corresponding cocycle $\beta_\sigma = \phi(\til
Z_{t_1^\sigma,t_1})$.  Then the difference of $\alpha$ and $\beta$ is
a coboundary, since
\begin{align*}
  \alpha_\sigma - \beta_\sigma &= \phi(\til Z_{t_0^\sigma, t_0} -
  \til Z_{t_1^\sigma,t_1}) \\
  &=
  \phi(\til Z_{t_0^\sigma,t_1^\sigma} - \til Z_{t_0,t_1}) \\
  &= \phi(\til Z_{t_0,t_1})^\sigma - \phi(\til Z_{t_0,t_1}).
\end{align*}
In fact, a similar calculation shows that the map
\[
\xymatrix@R-1pc{
  A^{(T,t_0,Z)}_{\bar K} \ar[r]& A^{(T,t_1,Z)}_{\bar K} \\
  P \ar@{|->}[r] & P + \phi(Z_{t_0}-Z_{t_1})
  }
\]
  is $\gal(K)$-equivariant on points (because we twist on the source by $\alpha$ and on the target by $\beta$), and thus descends to $K$ as an isomorphism $\epsilon_{t_0,t_1}$ which makes the following diagram commute:
  \begin{equation}
    \label{E:isomtorsor}
    \xymatrix@R-0.5pc{
      & A^{(T,t_0,Z)}\ar[dd]^{\epsilon_{t_0,t_1}}_\sim \\
    T \ar[ur]^{\psi_{(T,t_0,Z)}} \ar[dr]_{\psi_{(T,t_1,Z)}}\\
    & A^{(T,t_1,Z)}
    }
    \end{equation}

Finally, we verify that the operator $\tplus$ on cycles is compatible with addition in the Weil--Ch\^atelet group of torsors over $K$ under $A$.  Choose $t_j \in T_j(\bar K)$, and let $\alpha^{(j)}$ be the corresponding cocycle.  Then $A^{(T_1\times T_2, t_1\times t_2, Z_1\tplus Z_2)}$ is determined by the cocycle
\begin{align*}
  \sigma \mapsto & \ \phi((Z_1\tplus Z_2)_{(t_1,t_2)^\sigma} - (Z_1\tplus Z_2)_{(t_1,t_2)}) \\
         &= \phi( {Z_1}_{t_1^\sigma} + {Z_2}_{t_2^\sigma} - {Z_1}_{t_1} - {Z_2}_{t_2}) \\
    &    = \alpha^{(1)}_\sigma + \alpha^{(2)}_\sigma.
\end{align*}

  \end{proof}
  
  The isomorphism class of $A^{(T,Z)}$ depends only on the algebraic equivalence class of a fiber of $Z$:
  
  \begin{pro}
  \label{P:algequivtorsors}
  For $j=1,2$, let $T_j/K$ be a smooth geometrically connected variety, and let $Z_j \in \chow^i(T_j\times X)$.  Suppose that there are  $\bar K$-points $t_j \in T_j(\bar K)$ such that $(Z_1)_{t_1}$ and $(Z_2)_{t_2}$ are algebraically equivalent. Then there is an isomorphism of $K$-torsors $A^{(T_1, Z_1)} \iso A^{(T_2, Z_2)}$.
  \end{pro}
  
  \begin{proof}
In fact, by hypothesis every $\bar K$-fiber of $Z_1$ is algebraically equivalent to every $\bar K$-fiber of $Z_2$.  Therefore, in $\wc(A/K)$ we have an equality
\[
\null[A^{(T_1,Z_1)}] - \null[A^{(T_2,Z_2)}] = [A^{(T_1,Z_1)}] + [A^{(T_2,-Z_2)}] = [A^{(T_1\times T_2,Z_1\tplus (-Z_2))}] = [A].
\]
Here the first equality holds because $-Z_2$ has cocycle $-\alpha^{(2)}$, and the second is \Cref{T:main}(b).  For the third, write $W = Z_1 \tplus (-Z_2)$ and fix a $\bar K$-point $p$ of $T_1\times T_2$.  Every fiber of $W$ is algebraically trivial, so $\phi(W_p)$ is defined, and the cocycle of $(T_1\times T_2, p, W)$ is
\[
\sigma \mapsto \phi(W_{p^\sigma}) - \phi(W_p) = \phi(W_p)^\sigma - \phi(W_p),
\]
a coboundary.
\end{proof}

  \subsection{Beyond geometric connectedness}

  If the smooth parameter space $T$ is connected but not geometrically connected, 
  it seems
  unreasonable to expect that the theory of cycles will naturally
  induce a map from $T$ to a (torsor under an) abelian variety (e.g.,
  \cite[Cor.~1.5]{ACMValb}).  Nonetheless, we will see in this section that  a cycle class $Z
  \in \chow^i(T\times X)$ naturally induces a morphism to a $K$-scheme
  which, after base change, is isomorphic to a disjoint union of
  copies of $A$.

  Let $L\subseteq \bar K$
  be a finite separable extension of $K$.  (Note that $\bar K$ is also
  a separable closure of $L$.)  Then a regular homomorphism $\phi:\A^i(X_{\bar K})
  \to A(\bar K)$ induces a   regular homomorphism
  $\A^i((X_L)_{\bar K}) \to A_L(\bar K)$ over $L$ (\cite[Lemma
    2.2]{ACMVfunctor}).

  Now suppose $\tau \in \gal(K)$; let $L^\tau \coloneqq \tau(L)$.  We
  have an isomorphism
  \[
  \xymatrix{\gal(L) = \gal(\bar K/L) \ar[r]^\sim & \gal(\bar K/L^\tau)
    \\
    \sigma \ar@{|->}[r] & \sigma' \coloneqq \tau\inv \sigma \tau.
  }
  \]
If  $W$ is a scheme over $L$, set $W^\tau \coloneqq \tau^* W$; it is naturally a
  scheme over $L^\tau$.  Then -- remember we are letting Galois groups
  act on the right -- if $P^\tau \in W^\tau(\bar K)$ and $\sigma' \in
  \gal(L^\tau)$, then
  \[
  (P^\tau)^{\sigma'} = (((P^\tau)^{\tau\inv})^\sigma)^\tau =
  (P^\sigma)^\tau.
\]

  Let $B/L$ be an abelian variety, and let $\beta:\gal(L) \to B(\bar K)$ be a one-cocycle. Let $\beta^\tau$ be the one-cocycle $\gal(L^\tau) \to B^\tau(\bar K)$ which makes the following diagram commute:
  \[
  \xymatrix{
  \sigma \ar@{|->}[d] & \gal(L) \ar[d] \ar[r]^\beta & B(\bar K) \ar[d] & P \ar@{|->}[d] \\
  \sigma' \coloneqq \tau\inv \sigma \tau & \gal(L^\tau) \ar[r]^{\beta^\tau} & B^\tau(\bar K)& P^\tau
  }
  \]
By construction, for each $\sigma' \in \gal(L^\tau)$ we have
  \[
  (\beta^\tau)_{\sigma'} = (\beta_\sigma)^\tau.
  \]
By analyzing the induced actions of $\gal(L^\tau)$ on $B^\tau(\bar
K)$, we see that there is an isomorphism
  \begin{equation}
  \label{E:pullbacktorsor}
  B^\tau[\beta^\tau] \iso (B[\beta])^\tau
  \end{equation}
  of torsors over $L^\tau$ under $B^\tau$.  Indeed, if $P^\tau \in
  B^\tau(\bar K)$ and $\sigma' \in \gal(L^\tau)$, then the action of
  $\sigma'$ on $P^\tau$, twisted by $\beta^\tau$, is
  \[
(P^\tau)^{\til{\sigma'}} = (P^\tau)^{\sigma'} + (\beta^\tau)_{\sigma'}
  = (P^\sigma)^\tau+(\beta_\sigma)^\tau =(P^{\til\sigma})^\tau.
  \]

  \begin{lem}
  \label{L:LKpullback}
    Suppose that $T/L$ is a smooth geometrically connected scheme, and
    let $Z \in \chow^i(T \times_L (X_L))$.  Let $t_0 \in T(\bar K)$ be
    a $\bar K$-point.  Then
    \begin{alphabetize}
    \item $A^{(T^\tau, t_0^\tau, Z^\tau)} \iso (A^{(T,t_0,Z)})^\tau$;
              \item $\psi_{(T^\tau, t_0^\tau, Z^\tau)} =
                (\psi_{(T,t_0,Z)})^\tau$; and
                 \item if $t_1 \in T(\bar K)$ is a second geometric point, then $(\epsilon_{t_0,t_1})^\tau = \epsilon_{t_0^\tau, t_1^\tau}: A^{(T^\tau,t_0^\tau, Z^\tau)} \stackrel \sim\to A^{(T^\tau, t_1^\tau, Z^\tau)}$.
    \end{alphabetize}
  \end{lem}
 
  \begin{proof}
 The one-cocycle corresponding to the pulled-back data
    $(T^\tau, t_0^\tau, Z^\tau)$ is
    \begin{align*}
      \alpha(T^\tau, t_0^\tau, Z^\tau)_{\sigma'} & =
      \phi((Z^\tau)_{t_0^{\tau\sigma'}}- (Z^\tau)_{t_0^\tau}). \\
      \intertext{Because Galois acts on the right, we rewrite this as}
      &= \phi((Z^\tau)_{t_0^{\tau\tau\inv \sigma \tau}} -
      (Z_{t_0})^\tau) \\
      &= \phi((Z^\tau)_{t_0^{\sigma\tau}} - (Z_{t_0})^\tau) \\
      &= \phi((Z_{t_0^\sigma} - Z_{t_0})^\tau) \\
  &= \phi(Z_{t_0^\sigma} - Z_{t_0})^\tau \\
       &= (\alpha(T,t_0,Z)_\sigma)^\tau.
    \end{align*}
    Therefore $\alpha(T^\tau, t_0^\tau, Z^\tau) = \alpha(T,t_0,Z)^\tau$, and (a) follows from \eqref{E:pullbacktorsor}.
    
    For (b), it suffices to verify the commutativity of the diagram
    \[
    \xymatrix{T \ar[rr]^{\psi_{(T,t_0,Z)}} \ar[d]^{\tau^*} && A^{(T,t_0,Z)}\ar[d]^{\tau^*} \\
    T^\tau \ar[rr]^{\psi_{(T^\tau,t_0^\tau,Z^\tau)}} && A^{(T^\tau,t_0^\tau,Z^\tau)}
    }
    \]
    on $\bar K$-points.  Since 
    \[
    \tau^*\psi_Z(t) = \tau^*\phi(Z_t - Z_{t_0}) = \phi(Z_t-Z_{t_0})^\tau
    \]
    while
    \[
    \psi_{Z^\tau} (\tau^* t) = \psi_{Z^\tau}(t^\tau) = \phi(Z^\tau_{t^\tau} - Z^\tau_{t_0^\tau}),
    \]
    the claim (again) follows from the $\gal(K)$-equivariance of $\phi$.

    Part (c) is similar; one simply remembers that, on points, $\epsilon_{t_0,t_1}$ is translation by $\phi(Z_{t_0}-Z_{t_1})$.
  \end{proof}
  
  If $W$ is a scheme, let $\Pi_0(W)$ denote its set of irreducible components.  If $W$ is a scheme over $K$, let $\underline\pi_0(W)$ be its scheme of irreducible components.  It is an \'etale $K$-scheme, and $\underline\pi_0(W)(L)$ is functorially identified with $\Pi_0(W_L)$.
  
  \begin{pro}
  \label{L:geomdisconnected}
Suppose $T/K$ is a smooth connected scheme, and let $Z \in \chow^i(T\times X)$.
Let $t_0 \in T(\bar K)$ be a $\bar K$-point, and let $L \subseteq \bar K$ be a finite Galois extension of $K$ such that each irreducible component of $T_L$ is geometrically connected.
Then
\begin{alphabetize}
\item  There is a canonical  $L/K$ descent datum on
\begin{equation}
\label{E:unionAtwist}
\coprod_{V \in \Pi_0(T_L)} A^{(V, t_0\rest V, Z\rest V)};
\end{equation}
let $A^{(T,t_0,Z)}$ be the corresponding model over $K$.
\item The $L$-morphism
\begin{equation}
\label{E:unionpsitwist}
\xymatrixcolsep{5pc}\xymatrix{
T_L = \displaystyle{\coprod_{V \in \Pi_0(T_L)}} V \ar[r]^-{\psi_{(V, t_0\rest V, Z\rest V)}} & 
\displaystyle{\coprod_{V \in \Pi_0(T_L)}} A^{(V, t_0\rest V, Z\rest V)}}
\end{equation}
descends to a $K$-morphism
\[\xymatrix{
\psi_{(T,t_0,Z)} : T \ar[r] & A^{(T,t_0,Z)}.
}\]
\item If $t_1 \in T(\bar K)$ is a second geometric point, then there is a $K$-isomorphism $\epsilon_{t_0,t_1}: A^{(T,t_0,Z)} \stackrel\sim\to A^{(T,t_1,Z)}$ which makes the following diagram commute:
  \[
    \xymatrix@R-0.5pc{
      & A^{(T,t_0,Z)}\ar[dd]^{\epsilon_{t_0,t_1}}_\sim \\
    T \ar[ur]^{\psi_{(T,t_0,Z)}} \ar[dr]_{\psi_{(T,t_1,Z)}}\\
    & A^{(T,t_1,Z)}
    }
  \]
\end{alphabetize}
\end{pro}

\begin{proof}
Since $T$ is smooth and connected, it is irreducible.  Pulling $t_0$ back along $T_L \to T$ gives a $\bar K$-point on each irreducible component of $T_L$.  This point is not unique, and we fix one such choice $t_0\rest V$ for each component $V$.  Fix some irreducible component $W$ of $T_L$, and let $w_0 = t_0\rest W$.  Similarly, let $Y \in \chow^i(W\times X_L)$ be the restriction of $Z$ to $W\times X_L$.  

The Galois group $\gal(K)$ acts transitively on the set of irreducible components of $T_L$.  Let $H$ be the stabilizer of $W$ in $\gal(K)$, and let $R \subseteq \gal(K)$ be a system of representatives for the set of cosets $H\backslash \gal(K)$.  Then
\[
\Pi_0(T_L) = \st{ W^\tau: \tau \in R},
\]
and by \Cref{L:LKpullback}(a), the object defined in \eqref{E:unionAtwist} becomes
\[
\coprod_{\tau \in R} A^{(W^\tau, w_0^\tau, Y^\tau)} \iso \coprod_{\tau \in R} (A^{(W, w_0, Y)})^\tau,
\]
which descends to $K$.  Indeed, suppose $\sigma \in \gal(K)$.  Then multiplication by $\sigma$ defines a permutation of $R$, and thus an automorphism  $\tilde \sigma$ of $\coprod_{\tau\in R} (A^{(W,w_0,Y)})^\tau$ over $\sigma^*: \spec(\bar K) \to \spec(\bar K)$. Because $((W^\tau)^{\tilde \sigma})^{\tilde \rho} = (W^\tau)^{\widetilde{\sigma\rho}}$, this defines a descent datum on $\coprod_{\tau \in R} A^{(W^\tau,w_0^\tau, Y^\tau)}$; call the resulting scheme $A^{(T,t_0,Z)}$.

Similarly, let $\psi = \psi_{(W, w_0, Y)}$. Thanks to \Cref{L:LKpullback}(b), the morphism \eqref{E:unionpsitwist} is actually
\[
\xymatrix{
\coprod_{\tau \in R} W^\tau \ar[r]^-{\coprod_{\tau} \psi^\tau} & \coprod_{\tau \in R} (A^{(W,w_0,Y)})^\tau.
}
\]
The same argument we used to descend $A^{(T,t_0,Z)}$, when applied to the graph of $\coprod \psi^\tau$, shows that this morphism descends to $K$.

Finally, consider a second geometric point $t_1$ of $T$ and let $w_1 = t_1\rest W$.  Once again, because $(\epsilon_{w_0,w_1})^\tau = \epsilon_{w_0^\tau, w_1^\tau}$ (\Cref{L:LKpullback}(c)), $\coprod \epsilon_{w_0,w_1}^\tau$ descends to $K$.
\end{proof}

\begin{rem}
  In the setting of \Cref{L:geomdisconnected}, by construction, we have an isomorphism of \'etale $K$-schemes $\underline{\pi}_0(T) \iso \underline\pi_0(A^{(T,t_0,Z)})$.
\end{rem}

\subsection{Symmetric products}
\label{S:symmetric}

If $C/K$ is a smooth projective curve and $d$ is a positive integer, then there is a well-known Abel
map $\sym^{(d)}C \to \pic^d_{C/K}$.  We will see
(\Cref{C:chenrequest}) that this is a special case of the following
observation.

\begin{pro}
  \label{L:symmetric}
Let $T/K$ be a smooth geometrically connected scheme, suppose $Z \in \chow^i(T\times X)$, and let $t_0 \in T(\bar K)$.  Let $Z^{\tplus d} = Z\tplus Z \tplus \dots \tplus Z \in \chow^i(T^{\times d} \times X)$. Then

\begin{alphabetize}

\item the isomorphism class of the torsor $A^{(T^{\times d}, t_0^{\times d}, Z^{\tplus d})}$ is equal to the $d$-fold sum $d[A^{(T,t_0,Z)}] = [A^{(T,t_0,Z)}] + \dots + [A^{(T,t_0,Z)}]$; and
\item $\psi_{Z^{\tplus d}}$  factors through the symmetric product $\sym^{(d)}(T)$. In other words, there is a diagram of $K$-schemes 
\[
\xymatrix{
T^{\times d} = T\times \dots \times T \ar[rr]^{\psi_{Z^{\tplus d}}} \ar[rd]_{s_d} && \ \ A^{(T^{\times d},t_0^{\times d},Z^{\tplus d})} \\
& \sym^{(d)}(T) \ar@{-->}[ur]_{\psi_Z^{(d)}}
}\]
\end{alphabetize}
\end{pro}

\begin{proof}
  Part (a) is  a special case of \Cref{T:main}(b).  For part
  (b) since,  on $\bar K$-points, the fibers of $Z^{\tplus d}$ are constant on the fibers of the quotient map $T_{\bar K}^{\times d} \to \sym^{(d)}(T)_{\bar K}$, 
$\psi_{Z^{\tplus d}, \bar K}$ factors through some morphism $\psi^{(d)}_Z: \sym^{(d)}(T)_{\bar K} \to A_{\bar K}$. Now, $s_d$ is surjective on $\bar K$-points, and both $s_d$ and $\psi_{Z^{\tplus d}}$ are $\gal(K)$-equivariant on $\bar K$-points.
  It follows that $\psi^{(d)}$ is $\gal(K)$-equivariant, too, and thus descends to $K$.
\end{proof}

Suppose that $\dim T = 1$.  Then $Z^{\tplus d}$ descends to a cycle
class we abusively denote $\sym^{(d)}(Z)$ on $\sym^{(d)}(T)\times X$,
and by checking on $\bar K$-points we see that $\psi_Z^{(d)} =
\psi_{\sym^{(d)}(T), \sym^{(d)}(t_0), \sym^{(d)}(Z)}$.  (If $d \geq 2$ and $\dim
T>1$, then $\sym^{(d)}(T)$ is no longer smooth, and thus not a
suitable parameter space for cycle classes; but if one is willing to
work with the complement of the fat diagonal in $T^{\times d}$, and
choose a new base point accordingly, then a similar statement holds.)

As an example, now suppose that $X/K$ has dimension $n$.  The Albanese
torsor of $X$ is a torsor $\alb_{X/K}^{(1)}$ under an abelian variety
$\alb_{X/K}$ which is equipped with a morphism $\albmorphism: X \to
\alb_{X/K}^{(1)}$ which is initial for morphisms from $X$ to torsors
under abelian varieties.  (See \cite{ACMValb} for a recent treatment
of these objects and their history.)  The generalized difference map $\A^n(X_{\bar K}) \to \alb_{X/K}(\bar K)$ given by $\sum n_i P_i \mapsto \sum n_i \albmorphism(P_i)$ is a regular homomorphism,
and $\alb_{X/K}$ is actually the algebraic representative of $X$ in codimension
$n$.  (If $K$ is
algebraically closed, this observation is due to Murre \cite[\S 1.4
  and 1.8]{murre83}; for arbitrary $K$, see
\cite[Lem.~7.6]{ACMVfunctor}.)

The diagonal cycle $\Delta_X \in \chow^n(X\times X)$ determines a
torsor $\alb_{X/K}^{(X,\Delta_X)}$ under $\Ab^n_{X/K} \iso \alb_{X/K}$.
Choose a geometric base point $P \in X(\bar K)$.  By the universal
property of the Albanese torsor, $\psi_{(X,P, \Delta_X)}$ factors as
\[
\xymatrix{
&& \alb_{X/K}^{(1)}\ar@{-->}[d]^{\iota} \\
 X \ar[urr]^{\albmorphism} \ar[rr]^{\psi_{(X,P,\Delta_X)}} &&
 \alb_{X/K}^{(X,\Delta_X)}
}
\]
In fact, $\iota$ is the identity morphism.  It suffices to verify this
after base extension to a field of definition of $P$, at which point
the claim is contained in \cite[Lem.~7.6]{ACMVfunctor}.

For a positive integer $d$, let $\Delta_X^{(d)}$ denote the $d$-fold sum $\Delta_X \tplus \dots
\tplus \Delta_X$.  Then the isomorphism class of $\alb_{X/K}^{(d)} \coloneqq
\alb_{X/K}^{\Delta_X^{(d)}}$ is the same as the $d$-fold sum
$[\alb_{X/K}^{(1)}] + \dots + [\alb_{X/K}^{(1)}]$ (\cref{L:symmetric}(a)), and in this special case \cref{L:symmetric}(b)
is the assertion:

\begin{cor}
  \label{C:chenrequest}
  The morphism  $\psi_{\Delta_X^{(d)}}$ factors through the symmetric product; there is a diagram of $K$-varieties
\begin{equation}
  \label{E:descendtosym}
    \xymatrix{ X \times \dots \times X  \ar[rd]_{s_d} \ar[rr]^{\psi_{\Delta_X^{(d)}}} & &\alb^{(d)}_{X/K} \\
      & \sym^{(d)}(X) \ar[ru]_{\psi^{(d)}} &}
  \end{equation}
\end{cor}
If $X$ is a curve, then $\psi_{\Delta_X}$ is the Abel map, 
$\alb^{(1)}_{X/K} \iso \pic^1_{X/K}$, and more generally
$\alb^{(d)}_{X/K} \iso \pic^d_{X/K}$. \Cref{C:chenrequest} also appears as (the
existence part of) \cite[Cor.~3.2]{hassetttschinkel21}.  In contrast
to the argument given there, after the initial identification of
$\alb_{X/K}^{(1)}$ with $\Ab^n_{X/K}$, our proof is a simple
consequence of the theory of
regular homomorphisms.

\section{Application to rationality and cycles}
\label{S:rationality}

As we noted in the introduction,  torsors under algebraic intermediate
Jacobians (and algebraic representatives) have been used to study the
rationality of certain geometrically rational threefolds.   In this
section, we compare the torsors constructed here to those coming from
the Benoist--Wittenberg Chow scheme $\bwchow^2_{X/K}$, and explore the
extent to which they serve as a replacement in known rationality arguments.

\subsection{A framework for an obstruction to the existence of cycles defined over $K$}\label{S:BigAb}

It is well-known that for a smooth projective curve $C$ over a field
$K$,  the schemes $\pic^d_{C/K}$ provide obstructions to
the existence of effective $0$-cycles of degree $d$.  More precisely, consider
the Abel map 
$$\xymatrix{
\psi^{(d)} : \sym^{(d)}(C)\ar[r]& \pic^d_{C/K}}
$$
(\Cref{C:chenrequest}).
After base change to the separable closure, this is given by $D\mapsto \mathcal O_{C_{\bar K}}(D)$.  
The scheme $\pic^d_{C/K}$ is a torsor under $\pic^0_{C/K}$, and its class $[\pic^d_{C/K}]$ in the Weil--Ch\^atelet group $\wc(\pic^0_{C/K})$  provides an obstruction to the existence of an effective $0$-cycle of
degree $d$ on $C$, defined over~$K$.  Indeed, the existence of such a
cycle implies that $\pic^d_{C/K}$ admits a $K$-point
(namely, the image under $\psi^{(d)}$ of the $K$-point in $\sym^{(d)}(C)$ corresponding to the $0$-cycle), which in turn implies that $[\pic^d_{C/K}]=[\pic^0_{C/K}] = 0\in \wc(\pic^0_{C/K})$.   Taking a broader view, the subgroup $[\pic_{C/K}]$ of $\wc(\pic^0_{C/K})$ generated by the isomorphism classes of all the torsors $\pic^d_{C/K}$ provides a group that is an obstruction to the existence of $0$-cycles on $C$ defined over $K$.  

Our \Cref{T:main} allows one to construct similar obstructions for cycles of any dimension on any smooth projective variety over~$K$.  
For instance,  given a regular homomorphism $\phi\colon \A^i(X_{\bar
  K}) \to A(\bar K)$ over $K$ and any family of cycle classes $Z\in
\chow^i(T\times X)$  parameterized by a smooth variety $T$ over $K$,  
the class $[A^{(T,Z)}]$ in $\wc(A)$, if nontrivial, is an obstruction to the existence of a $K$-point of $T$.
   Moreover,  if the $K$-morphism $T \to A^{(T,Z)}$ from \Cref{T:main} is an
  isomorphism, then $T$ has a $K$-point if and only if $[A^{(T,Z)}]=0$
  in $\wc(A/K)$.  This is the obstruction used in the applications below. 
  
 Bearing this in mind, consider now  the  algebraic representative $\phi:\A^i(X_{\bar K}) \to \Ab^i_{X/K}(\bar K)$, if it exists (e.g., $i=1,2,\dim X$).  The subgroup  of $\wc(\Ab^i_{X/K})$ generated by the isomorphism classes of the torsors $\Ab^{i,(T,Z)}_{X/K}$, as $(T,Z)$ ranges over all families of codimension-$i$ cycle classes on $X$ (as in the first paragraph), provides a group that  can be viewed as an obstruction to the existence of certain cycles on $X$ defined over $K$.

\subsection{Comparison with the Chow scheme}

We now introduce a more restricted setting where it is possible to compare the present method with that of Benoist--Wittenberg.  Let $X/K$ be a smooth projective geometrically rational threefold over a field.  Benoist and Wittenberg are able to construct
  a codimension-2 Chow scheme $\bwchow^2_{X/K}$
 \cite[Thm.~3.1]{benoistwittenberg23}.    Since
 $\bwchow^2_{X/K}$ is a group scheme, any geometrically irreducible
 component $\bwchow^{2,\xi}_{X/K}$ is naturally a torsor over $K$
 under the connected component of the identity $\bwchow^{2,\circ}_{X/K}$.  Moreover, the group structure on
 $\bwchow^2_{X/K}$, and thus on its group of connected components,
 implies the equality $[\bwchow^{2,\xi}_{X/K}] +
 [\bwchow^{2,\xi'}_{X/K}] = [\bwchow^{2,\xi+\xi'}_{X/K}]$ in
 $\wc(\bwchow^{2,\circ}_{X/K})$.

 We now suppose that $K$ is perfect; then $\bwchow^{2,\circ}_{X/K}$ is
 canonically isomorphic to $\Ab^2_{X/K}$, and thus geometrically
 irreducible components of $\bwchow^2_{X/K}$ are torsors under the
 algebraic representative.

To compare the torsors which arise, suppose $T/K$ is smooth and
geometrically integral, and that $Z \in \chow^2(T\times X)$.  Fix a
geometric point $t_0 \in T(\bar K)$.  On one hand, \Cref{T:main}
defines a torsor $\Ab^{2,(T,Z)}_{X/K}$ under $\Ab^2_{X/K}$, and a
$K$-rational morphism $\psi_{(T,t_0,Z)}: T \to \Ab^{2,(T,Z)}_{X/K}$.
On the other hand, the image of the classifying map $\chi_{(T,Z)}:T \to
\bwchow^2_{X/K}$ is supported in some geometrically irreducible
component $\bwchow^{2,(T,Z)}_{X/K}$, also a torsor under
$\Ab^2_{X/K}$.

\begin{lem}
\label{L:comparison}
Let $X/K$ be a smooth projective geometrically rational threefold over a perfect field.
Let $T/K$ be a smooth and geometrically integral scheme, $Z \in \chow^2(T\times X)$, and $t_0 \in T(\bar K)$.  Then there is an isomorphism $\iota$ which makes the following diagram commute:
\[
\xymatrix@R-1pc{
 & \Ab^{2,(T,Z)}_{X/K} \ar[dd]^\iota_\sim \\
T \ar[ur]^{\psi_{(T,t_0,Z)}} \ar[dr]_{\chi_{(T,Z)}} \\
& \bwchow^{2,(T,Z)}_{X/K}}
\]
In particular, $\Ab^{2,(T,Z)}_{X/K}$ and $\bwchow^{2,(T,Z)}_{X/K}$ are isomorphic.
\end{lem}

\begin{proof}
Initially, we work over $\bar K$.  Let
 $Z' = Z_{\bar K} - T_{\bar K} \times Z_{t_0}\in \chow^2(T_{\bar K} \times X_{\bar K})$, and note that all fibers of $Z'$ over $T_{\bar K}$ are algebraically trivial.  We claim there is a commutative diagram
\[
\xymatrix{
 &(\Ab^{2,(T,Z)}_{X/K})_{\bar K} \ar[r]^\sim & (\Ab^2_{X/K})_{\bar K} \ar[dd]^\sim\\
T_{\bar K} \ar[urr]_{\psi_{(T,t_0,Z')}} \ar[ur]^{\psi_{(T,t_0,Z), \bar K}}
\ar[dr] _{\chi_{(T,Z),\bar K}}
\ar[drr]^{\chi_{(T,Z')}} \\
& (\bwchow_{X/K}^{2,(T,Z)})_{\bar K} \ar[r]_\sim & (\bwchow_{X/K}^{2,\circ})_{\bar K}
}\]
where the top horizontal isomorphism is the trivialization $\psi_{(T,t_0,Z)}(t_0)\mapsto 0_{\Ab^2_{X/K}, \bar K}$; the bottom horizontal isomorphism is translation by $\chi_{(T,Z)}(t_0) \in \bwchow^2_{X/K}(\bar K)$; and the vertical isomorphism is supplied by \cite[Thm.~3.1(vi)]{benoistwittenberg23}.  Indeed, because $Z'$ has algebraically trivial fibers, the outer triangle commutes \cite[Thm.~3.1(iv)]{benoistwittenberg23}.  The commutativity of the top triangle is verified by looking at $\bar K$-points, while the commutativity of the bottom triangle follows from the definition of the group functor which $\bwchow^2_{X/K}$ represents.

Consequently, there is an isomorphism $\iota_{\bar K}$ which makes the following diagram of $\bar K$-schemes commute: 
\begin{equation}
\label{E:iotaKbar}
\xymatrix@R-1pc{
 &(\Ab^{2,(T,Z)}_{X/K})_{\bar K}\ar[dd]^\sim \\
T_{\bar K}\ar[ur]^{\psi_{(T,t_0,Z), \bar K}}
\ar[dr] _{\chi_{(T,Z),\bar K}}\\
& (\bwchow_{X/K}^{2,(T,Z)})_{\bar K} 
}\end{equation}
Possibly after replacing $T$ with a larger parameter space, we may and
will assume $\psi_{(T,t_0,Z)}$ (and thus $\chi_{(T,Z)}$) surjective.
To see that this is possible, let $M\in \chow^2(\Ab^2_{X/K}\times X)$ be
 a miniversal cycle, i.e., a family of
algebraically trivial cycle classes on $X$ parametrized by
$\Ab^2_{X/K}$ such that the induced morphism $\psi_{(\Ab^2_{X/K},
  0_{\Ab^2_{X/K}}, M)}: \Ab^2_{X/K} \to \Ab^2_{X/K}$ is an isogeny \cite[Lem.~4.7]{ACMVfunctor}.  In particular, $\psi_{(\Ab^2_{X/K},0_{\Ab^2_{X/K}},M)}$ is surjective, and the fiber $M_{0_{\Ab^2_{X/K}}}$ is the trivial cycle class. Now let $T'' = T\times \Ab^2_{X/K}$, and let $Z'' = Z \tplus M$.  Identify $T$ with $T\times \st{0_{\Ab^2_{X/K}}}$ as a closed subscheme of $T''$; then $\psi_{(T'', t_0 \times 0_{\Ab^2_{X/K}}, Z'')}$ is surjective, and its restriction to $T$ is $\psi_{(T,t_0,Z)}$.  This replacement does not change the torsor.  Indeed, the cocycle of $(\Ab^2_{X/K}, 0_{\Ab^2_{X/K}}, M)$ is trivial, because $0_{\Ab^2_{X/K}}$ is a $K$-point.  We relabel $(T'', t_0 \times 0_{\Ab^2_{X/K}}, Z'')$ as $(T, t_0, Z)$ and proceed.

Using the universal property of the Albanese torsor of $T$, the morphisms $\psi_{(T,t_0,Z)}$ and $\chi_{(T,Z)}$ factor through $\alb^{(1)}_{T/K}$.  (Here $T$ need not be proper; see \cite{ACMValb} for the Albanese torsor of an incomplete variety.)  In particular, we have a surjection of torsors $\alb^{(1)}_{T/K} \to \Ab^{2,(T,Z)}_{X/K}$ over a surjection of abelian varieties $\bar\psi: \alb_{T/K} \to \Ab^2_{X/K}$, and a surjection of torsors $\alb^{(1)}_{T/K} \to \bwchow^{2,(T,Z)}_{X/K}$ over a surjection of abelian varieties $\bar\chi: \alb_{T/K} \to \bwchow^{2,\circ}_{X/K}$.  Let $G = \ker \bar\psi$ and $H = \ker\bar\chi$.  Then our torsors may be constructed as quotients
\[
\Ab^{2,(T,Z)}_{X/K} \iso \alb^{(1)}_{T/K}/G\text{ and }
\bwchow^{2,(T,Z)}_{X/K} \iso \alb^{(1)}_{T/K}/H.
\]
Using the commutativity of \eqref{E:iotaKbar}, and the fact that the morphisms from $T_{\bar K}$ (still) factor through $\alb_{T_{\bar K}/\bar K}$, we find that $G_{\bar K} = H_{\bar K}$.  Therefore, $G= H$, and $\iota_{\bar K}$ descends to $K$.
\end{proof}

\subsection{An application to rationality}       

We close this paper by observing that, at least over a perfect field,  the
theory of torsors under the algebraic representative is an
adequate substitute for torsors coming from the Chow scheme.  In the
following argument, we appeal to \cite{benoistwittenberg23} for
(sometimes classical) geometric facts, but systematically use the
torsors constructed in \Cref{T:main}.

\begin{teo}[Benoist--Wittenberg, Hassett--Tschinkel] \label{T:BW}
	Assume $K$ is perfect, and let $X\subseteq \proj^5_K$ be a smooth
	complete intersection of two quadrics.  Then $X$ is rational if and
	only if it contains a $K$-rational line.
\end{teo}
\begin{proof}

To start with, we work over an arbitrary field $K$.
If $X$ contains a line defined over $K$, then projection from the line shows that $X$ is rational over~$K$.
More precisely,   by the purely geometric  \cite[Prop.~4.3]{benoistwittenberg23}, projection from a $K$-line $\Lambda \subset X$ gives a diagram
\begin{equation}
\label{E:blowup}
\begin{tikzcd}
 & X' \arrow[dl, "\mu"'] \arrow[dr, "\nu"] & \\
 X & & \proj^3  \end{tikzcd}
\end{equation}
where $\mu$ is the blow-up of $X$ along $\Lambda$ and $\nu$ is the blow-up of $\proj^3$ along the discriminant locus $\Delta \subset \proj^3$ which is a smooth, geometrically connected curve of genus 2. Denote by $E$ the exceptional divisor of $\nu$.

The converse was already proven in \cite[Thm.~6.5]{hassetttschinkel21} if $K\subseteq \cx$, and in 
\cite[Thm.~4.7]{benoistwittenberg23} over an arbitrary field $K$.  
  Our modest aim is to sketch how the geometric arguments
used and established in \cite{benoistwittenberg23} can be used to give a proof using torsors under algebraic representatives in place of the Chow scheme of \cite{benoistwittenberg23}, at least in the case $K$ is perfect. 
For the reader's convenience, our notation remains consistent with that of \cite{benoistwittenberg23} wherever feasible.

 Let $F$ be the variety of lines contained in $X$, and let $G$ be the variety of conics contained in~$X$; these are smooth projective varieties over $K$. 
 From \Cref{T:main}, the universal line induces a $K$-morphism $a\colon F \rightarrow P$, where $P$ is a torsor under the algebraic representative $\Ab^2_X$, and since a smooth conic on $X$ admits a degeneration to two lines,
 the universal conic induces a $K$-morphism $\psi\colon G \rightarrow P^2$, where
       $P^2$ is the torsor under $\Ab^2_X$ that is the Baer sum of $P$ with itself.  (For the latter, combine \Cref{P:algequivtorsors} with \Cref{T:main}(b).)
       
  In the first 3 steps below, we work over an arbitrary field $K$, while in Step 4 we assume $K$ to be perfect. We refer to the recent survey \cite{wittenberg26} for a concise account of the subtleties involved in dealing with Step 4 in the imperfect case.
         \medskip

\noindent      \textbf{Step 1: The morphism $a\colon F \rightarrow P$ is an isomorphism.} 
      This is the analogue in the context of algebraic representatives of \cite[Thm.~4.5(ii)]{benoistwittenberg23}. 
     It is enough to show the
      statement after a finite Galois extension of $K$.
     Since $F$ is a nonempty smooth variety over $K$, it has a closed point with separable residue field.  We may and do assume $F(K) \neq \emptyset$.
            
 Consider then the blow-up diagram \eqref{E:blowup} above.
     The blow-up formula for Chow groups gives, after restriction to algebraically trivial cycles, 
     an isomorphism    $\mu_*\nu|^*_E\colon \chow^1(\Delta_{\bar K})_{\text{alg}} \stackrel{\sim}{\longrightarrow} \chow^2(X_{\bar K})_{\text{alg}} $ with inverse induced by a correspondence over $K$. 
     Hence, the correspondence $\mu_*\nu|^*_E$ induces an isomorphism
     $h\colon \pic^0_\Delta \stackrel{\sim}{\longrightarrow} \Ab^2_X$ of abelian varieties.
     
      Let $b\colon \Delta \rightarrow F$ be the $K$-morphism assigning to $x \in \Delta$ the line $\mu(\nu^{-1}(x))$; it is well-defined, as it assigns the residual line meeting $\Lambda$. This induces a morphism $g\colon \pic^0_\Delta \rightarrow \alb_F$.
      On the other hand, the morphism $a\colon F \rightarrow P$ induces a map $f\colon \alb_F \rightarrow \Ab^2_X$. 
      We claim that the composition $f \circ g\colon \pic^0_\Delta \rightarrow \Ab^2_X$ coincides with $h$. 
      Indeed, $f \circ g$ is induced by $a \circ b\colon \Delta \to F \rightarrow P$. 
      Consequently, it is induced by the correspondence
      $\mu_*\nu|^*_E$.
      This is exactly the correspondence inducing $h$ via the blow-up formula. 
      Hence, $f \circ g = h$. 
      Since $h$ is an isomorphism,  $g$ is injective. 
      Therefore, $F$ has first Betti number at least 4.
          Now, we know from \cite[Lem.~4.1(iii)]{benoistwittenberg23} that $F$ is a geometrically connected smooth projective surface with trivial canonical bundle (in particular, the surface $F$ is geometrically minimal).
      By the classification of minimal surfaces of Kodaira dimension 0, $F_{\bar K}$ must be an abelian variety, and so $g$ is
      an isomorphism.
           It immediately follows that $f$ is an isomorphism, and therefore $a$ is an isomorphism.
\medskip

	As a consequence, we note that in our specific geometric situation the algebraic representative $\Ab^2_X$ 
	satisfies base-change along any field extension, i.e., there is a natural $L$-isomorphism $\Ab^2_{X_L} \stackrel{\sim}{\longrightarrow} (\Ab^2_X)_L $ for any field extension $L/K$. 
This is \emph{a priori}  only known for separable extensions of fields; cf.\ \cite{ACMVfunctor}.
\medskip

     \noindent \textbf{Step 2: The morphism $\psi\colon G \rightarrow P^2$ has image a curve $D$ and  $\pic^1_D \rightarrow P^2$ is an isomorphism.}       
      Still assuming $F(K) \neq \emptyset$, consider the map $c\colon
      \Delta \rightarrow P^2$ given by $x \mapsto a(b(x)) +
      a(\Lambda)$; it is induced by the correspondence $\mu_*\nu|^*_E + \Delta \times_K \Lambda$.
                    Since $f \circ g = h$ is an isomorphism on the underlying abelian varieties, $c$ induces an isomorphism $\pic^1_\Delta \stackrel{\sim}{\longrightarrow} P^2$. 
      In particular, $c$ is a closed immersion. 
      Let us denote its image by $D = c(\Delta) \subseteq P^2$.
       Arguing as in the proof of \cite[Thm.~4.5(iii)]{benoistwittenberg23} (which is purely geometric and does not involve the Chow scheme), we have $D({K}^{\text{alg}}) = \psi(G({K}^\text{alg}))$, where ${K}^{\text{alg}}$ is the algebraic closure of $K$.
      It follows that the subvariety $D \subseteq P^2$ descends to $K$.
      We thus get a $K$-isomorphism $\pic^1_D
      \stackrel{\sim}{\longrightarrow} P^2$, inducing a $K$-isomorphism of
  abelian varieties $\pic^0_D \stackrel{\sim}{\longrightarrow} \Ab^2_X$,  without needing to assume $F(K) \neq \emptyset$. In characteristic not 2, this result goes back to Wang~\cite{wang}.
      \medskip

       \noindent \textbf{Step 3: The $K$-isomorphism $\pic^0_D \stackrel{\sim}{\longrightarrow} \Ab^2_X$ is compatible with the principal polarizations.} 
       Recall that for a smooth projective curve $C$ over $K$, the theta divisor endows $\pic^0_C$ with a principal polarization.
       Let $X$ be a smooth projective threefold over a perfect field $K$, and assume either that $X$ is geometrically rational, or that  $X$ can be lifted to a smooth projective threefold $Y$ in characteristic zero with $\chow_0(Y)\otimes \rat $ universally trivial and that $\chow_0(X_{K^{\text{alg}}})$ is universally trivial. Then, by \cite[Cor.~2.8]{benoistwittenberg20} and \cite[Cor.~13.3]{ACMVdiag} respectively, the algebraic representative 
       $\Ab^2_{X}$ is endowed with a canonical principal polarization. 
       If now $X$ is the smooth complete intersection of two quadrics in $\proj ^5_K$ with $K$ an arbitrary field, then as noted at the end of Step 1, there is a natural isomorphism $\Ab^2_{X_{K^{\text{alg}}} } \to (\Ab^2_X)_{K^{\text{alg}}}$. 
       Hence, denoting by $K^{\text{perf}}$ the perfect closure of $K$ and using the fact that Hom-schemes between abelian varieties over $K$ are \'etale over $K$, the canonical principal polarization on  $\Ab^2_{X_{K^{\text{perf}}}}$ descends uniquely to a principal polarization on $\Ab^2_X$.

       Now, by \cite[Lem.~2.10]{benoistwittenberg20}, the isomorphism $h\colon \pic^0_\Delta \stackrel{\sim}{\longrightarrow} \Ab^2_X$ of Step 1 (which is defined after some finite Galois extension of $K$) induced by the blow-up formula is compatible with the principal polarizations. 
       It follows that the $K$-isomorphism $\pic^0_D \stackrel{\sim}{\longrightarrow} \Ab^2_X$ is also compatible with the principal polarizations, as after base-change to a finite Galois extension it is obtained by precomposing~$h$ with the principally polarized isomorphism $\pic^0_D \cong \pic^0_{\Delta}$ induced by the isomorphism~$\Delta \cong D$.
            \medskip
      
       \noindent \textbf{Step 4: If $K$ is perfect and if $X$ is rational, then $P \cong \pic^d_D$ for some $d\in \integ $.}   
    This relies on the following key observation: If $X$ is a rational 3-fold over a perfect field $K$, then there exists a smooth projective curve $B$ over $K$ such that the torsor $P$ is a principally polarized direct factor of $\pic_B$. 
       This is the analogue of \cite[Thm.~3.1(vii)]{benoistwittenberg23} and can be seen as follows. 
       By resolution of indeterminacy for threefolds \cite{abhyankar},
       there exists a proper birational morphism $\pi \colon \tilde{X}\to X$
       of smooth projective varieties, together with a chain of blow-ups along smooth centers $\tilde{X} = Y_N \to \cdots \to Y_1 \to Y_0=\proj ^3_K$.
       By the projection formula \cite[Lem.~2.9]{benoistwittenberg20},  $\Ab^2_{X}$ is a direct summand of  $\Ab^2_{\tilde{X}}$ as a principally polarized abelian variety,
       and by the blow-up formula \cite[Lem.~2.10]{benoistwittenberg20}, $\Ab^2_{\tilde{X}}$ is isomorphic to $\pic^0_B$ as a principally polarized abelian variety, where $B$ is the disjoint union of the smooth projective curves being blown up in the successive $Y_j$. 
       Since $\Ab^2_{X}$ is then realized as a direct summand of $\pic^0_B$ via the action of correspondences, 
       we find by applying \cref{T:main}(a) to the universal line $\mathcal L \in \chow^2(F\times_KX)$ 
        that $P$ is a direct summand of $\pic_B$.

       We then argue as in \cite[Thm.~3.11(iii)]{benoistwittenberg23}. All indecomposable principally polarized direct factors of $\pic^0_B$ are of the form $\pic^0_{B'}$ for some connected component $B'$ of $B$. 
       Let $B'$ be the component corresponding to $\Ab^2_X$.
       By Step 3, we have a principally polarized isomorphism $\Ab^2_X \cong \pic^0_D$. 
       Since $D$ is geometrically connected of genus $\geq 2$, it follows that $\pic^0_D \cong \pic^0_{B'}$ is geometrically indecomposable of dimension $\geq 2$. Hence $B'$ is geometrically connected of genus $\geq 2$. 
       By the Torelli theorem for curves, we can identify $D$ and $B'$ in such a way that the split surjection
       $\pic^0_B \to \pic^0_D \cong \Ab^2_X$ is induced by the inclusion $D\cong B' \hookrightarrow B$.
  As such, we get $P \cong \pic^d_D$ for some $d\in \integ $. 
\medskip

\noindent \textbf{Conclusion.}   Subtracting the torsor isomorphisms of Step 2 and Step 4 in the Weil--Ch\^atelet group gives the identity:
$$
[P]=[\pic^1_{D/K}] - [\pic^d_{D/K}] = [\pic^{1-d}_{D/K}].
$$
Since $D$ has genus $2$, it has a degree-$2$ zero-cycle defined over $K$, namely its canonical divisor, and so we have that  $[\pic^{2n}_{D/K}]= 0$ for any  integer $n$. 
As one of $d$ and $1-d$ is even, we must have that $[P]=0$, and so $P(K) \neq \emptyset$. By Step 1, we find that $F(K)\neq \emptyset$, and we are done.
  \end{proof}

\begin{rem}
In fact, an elementary argument lets us extend Theorem~\ref{T:BW} to the case where $K$ is imperfect of odd characteristic.

For an arbitrary field $K$, and with no rationality hypothesis, we have $4[P] = 0$.
 Indeed, Step 2 gives $\pic^1_D \cong P^2$, i.e. $[\pic^1_D] = 2[P]$; and we have seen that since $D$ has genus 2,  $\omega_D$ is a $K$-rational divisor class of degree 2 and  thus $\pic^2_D(K) \neq \emptyset$. Hence $2[\pic^1_D] = [\pic^2_D] = 0$, and therefore $4[P] = 0$. In particular, the period $\operatorname{per}(P)$ -- that is, the order of $[P]$ in $\wc(K)$ -- divides $4$.

Now assume that $K$ has characteristic $p>0$ and that $X$ is rational over $K$. Invoking Theorem~\ref{T:BW} for the perfect field $K^{\mathrm{perf}}$ shows that $X_{K^{\mathrm{perf}}}$ contains a $K^{\mathrm{perf}}$-line, i.e., $F(K^{\mathrm{perf}}) \neq \emptyset$. Since $F$ is of finite type over $K$, that point is already defined over a finite purely inseparable extension of $K$.  In particular, the index $\operatorname{ind}(P)$ -- that is, the greatest common divisor of the degrees of all finite extensions of $K$ which trivialize $P$ -- divides some power of $p$.

Now, $\operatorname{per}(P)$ and $\operatorname{ind}(P)$ have the same prime divisors \cite[Prop.~5]{langtate}.  So if $p\neq 2$, then $\operatorname{per}(P) = 1$; $P(K) \neq \emptyset$; and Step 1 lets us conclude that $F(K)\neq \emptyset$, i.e., that $X$ contains a $K$-rational line.
\end{rem}

       \bibliographystyle{amsalpha}
       \bibliography{DCG}

\end{document}